\documentclass[a4paper,11pt]{amsart}
\usepackage{amsfonts,amsmath,amsthm,amssymb}
\usepackage{amsaddr}
\usepackage{enumitem}
\usepackage[fit]{truncate}
\usepackage{cleveref}
\usepackage{fullpage}

\theoremstyle{plain}
\newtheorem{theorem}{Theorem}[section]

\newtheorem{lemma}[theorem]{Lemma}

\newtheorem{proposition}[theorem]{Proposition}

\theoremstyle{definition}

\begin{document}
\title{On the $3$-adic valuation of the Narayana numbers}
\author{Russelle Guadalupe}
\address{Institute of Mathematics, University of the Philippines-Diliman\\
	Quezon City 1101, Philippines}
\email{rguadalupe@math.upd.edu.ph}
\date{}
\maketitle

\begin{abstract}
	We define the Narayana sequence $\{a_n\}_{n\geq 0}$ as the one satisfying the linear recurrence relation $a_n = a_{n-1}+a_{n-3}$ for $n\geq 3$, with initial values $a_0 = 0$ and $a_1 = a_2=1$. In this paper, we fully characterize the $3$-adic valuation of ${a_n}$ and use this to determine all Narayana numbers that are factorials.  
\end{abstract}

\section{Introduction}

The arithmetic properties of the linear recurrence sequences are one of the topics in number theory that are extensively studied. One example concerns about the Fibonacci sequence $\{F_n\}_{n\geq 0}$ defined by $F_n = F_{n-1}+F_{n-2}$ for $n\geq 2$, with initial values $F_0=0$ and $F_1=1$. The first few terms of this sequence are
\[0,1,1,2,3,5,8,13,21,34,55,89,144,233,377,610,987,\ldots\]
For a prime $p$, the $p$-adic valuation (or the $p$-adic order) $\nu_p(r)$ is the exponent of the highest power of $p$ which divides $r$. The $p$-adic valuation of the Fibonacci numbers was fully characterized (see \cite{halton, lengyel, lengyel2, robinson, vinson}). In particular, Lengyel \cite{lengyel} showed that for $n\geq 1$,
\[\nu_2(F_n)=\begin{cases}
0, &\text{ if }n\equiv 1,2\pmod 3;\\
1, &\text{ if }n\equiv 3\pmod 6;\\
3, &\text{ if }n\equiv 6\pmod {12};\\
\nu_2(n)+2, &\text{ if }n\equiv 0\pmod {12}\\
\end{cases}\]
using congruence properties involving $F_n$ (see \cite{jacobson, nied}). However, the behavior of the $p$-adic valuation of linear recurrence sequences of higher order is much less known. A particular case involves a well-known generalization of the Fibonacci numbers, the Tribonacci sequence $\{T_n\}_{n\geq 0}$ defined by $T_n = T_{n-1}+T_{n-2}+T_{n-3}$ for $n\geq 3$, with initial values $T_0=0$ and $T_1=T_2=1$. The first few terms of this sequence are
\[0,1,1,2,4,7,13,24,44,81,149,274,504,927,1705,\ldots\] 
In 2014, Marques and Lengyel \cite{marleng} proved that for $n\geq 1$,
\[\nu_2(T_n)=\begin{cases}
0, &\text{ if }n\equiv 1,2\pmod 4;\\
1, &\text{ if }n\equiv 3,11\pmod {16};\\
2, &\text{ if }n\equiv 4,8\pmod {16};\\
3, &\text{ if }n\equiv 7\pmod {16};\\
\nu_2(n)-1, &\text{ if }n\equiv 0\pmod {16};\\
\nu_2(n+4)-1, &\text{ if }n\equiv 12\pmod {16};\\
\nu_2((n+1)(n+17))-3, &\text{ if }n\equiv 15\pmod {16}
\end{cases}\]
and used the $2$-adic valuation of $T_n$ to show that $1,2$ and $24$ are the only Tribonacci numbers that are factorials. Since then, several authors have worked on the $2$-adic valuation of the generalized Fibonacci numbers (see \cite{buntonien, marleng2, sobol, young}). Another example is about the Tripell sequence $\{t_n\}_{n\geq 0}$ defined by $t_n = 2t_{n-1}+t_{n-2}+t_{n-3}$ for $n\geq 3$, with initial values $t_0=0, t_1=1$ and $t_2=2$. The first few terms of this sequence are
\[0,1,2,5,13,33,84,214,545,1388,3535,9003,22929,58396,\ldots\]
In 2020, Bravo, D\'{i}az, and Ram\'{i}rez \cite{bravo} proved that for $n\geq 1$,
\[\nu_3(t_n)=\begin{cases}
0, &\text{ if }n\equiv 1,2,3,4\pmod 6;\\
\nu_3(n), &\text{ if }n\equiv 0\pmod 6;\\
\nu_3(n+1), &\text{ if }n\equiv 5\pmod 6
\end{cases}\]
and applied the $3$-adic valuation of $t_n$ to show that $1$ and $2$ are the only Tripell numbers that are factorials.\\
\\
Recall that the Narayana sequence $\{a_n\}_{n\geq 0}$ is defined by $a_n = a_{n-1}+a_{n-3}$ for $n\geq 3$, with initial values $a_0 = 0$ and $a_1 = a_2=1$. The first few terms of this sequence are
\[0, 1, 1, 1, 2, 3, 4, 6, 9, 13, 19, 28, 41, 60, 88, 129,189,277,\ldots\]
Some properties of Narayana numbers and its generalizations can be found in \cite{allouche, ballot, bilgici}. \\
\\
In this paper, we apply Zhou's \cite{zhou} method of constructing identities of linear recurrence sequences to deduce several congruence properties involving $a_n$ and to prove our main result, which fully describes the $3$-adic valuation of $a_n$.
\begin{theorem}
	\label{th:thm1}
	For integers $n\geq 1$, we have
	\[\nu_3(a_n) = \begin{cases}
	0, &\text{ if }n\equiv 1,2,3,4,6\pmod 8;\\
	1, &\text{ if }n\equiv 5,7,13,15\pmod {24};\\
	2, &\text{ if }n\equiv 8\pmod {24};\\
	\nu_3(n+1)+1, &\text{ if }n\equiv 23\pmod {24};\\
	\nu_3(n+3)+1, &\text{ if }n\equiv 21\pmod {24};\\
	\nu_3(n)+2, &\text{ if }n\equiv 0\pmod {24};\\
	\nu_3(n+8)+2, &\text{ if }n\equiv 16\pmod {24}.\\
	\end{cases}\] 
\end{theorem}
One application of the $p$-adic valuation of linear recurrence sequences is to give upper bounds of the solutions of the Diophantine equations involving factorials and these sequences. In this paper, we use \Cref{th:thm1} to determine all Narayana numbers that are factorials, as shown by the following result.
\begin{theorem}
	\label{th:thm2}
	The only solutions to the Diophantine equation $a_n = m!$ in positive integers $n$ and $m$ are
	$(n,m)\in \{(1,1),(2,1),(3,1),(4,2),(7,3)\}$.
\end{theorem}

\section{Preliminaries}
In this section, we present some preliminary results that are used in the proofs of our main theorems. We begin with the bounds of the $p$-adic valuation of a factorial, which is a consequence of Legendre's formula (see \cite{legendre}).

\begin{lemma}
\label{le:lem1}
For any integer $m\geq 1$ and prime $p$, we have
\[\dfrac{m}{p-1}-\left\lfloor \dfrac{\log m}{\log p}\right\rfloor-1\leq \nu_p(m!)\leq \dfrac{m-1}{p-1},\]
where $\lfloor x\rfloor$ denotes the largest integer less than or equal to $x$.
\end{lemma}

\begin{proof}
See \cite[Lem. 2.4]{marques}.
\end{proof}

We next have the exponential growth of the Narayana sequence $\{a_n\}$.

\begin{lemma}
\label{le:lem2}	
For all integers $n\geq 1$, we have $\alpha^{n-3}\leq a_n\leq \alpha^{n-1}$, where $\alpha > 1$ is the real root of the characteristic
polynomial $f(x) :=x^3-x^2-1$ given by
\[\alpha = \dfrac{1}{3}\left(1+\sqrt[3]{\dfrac{29-3\sqrt{93}}{2}}+\sqrt[3]{\dfrac{29+3\sqrt{93}}{2}}\right)\approx 1.4656.\]
\end{lemma}

\begin{proof}
A direct application of the rational roots theorem shows that $f(x)$ is irreducible over $\mathbb{Q}$. Moreover, it has a real root $\alpha > 1$ and two conjugate complex roots lying inside the unit circle. We now use induction on $n$. We note that the statement holds for $n=1$ as $0.6823 < \alpha^{-2}\leq a_1=1 \leq \alpha^0$. We now suppose that $\alpha^{m-3}\leq a_m\leq \alpha^{m-1}$ holds for all integers $1\leq m\leq n-1$. Using the recurrence formula for $a_n$, we see that
\[\alpha^{n-4}+\alpha^{n-6}\leq a_n\leq \alpha^{n-2}+\alpha^{n-4}\implies \alpha^{n-6}(\alpha^2+1)\leq a_n\leq\alpha^{n-4}(\alpha^2+1).\]
Since $\alpha^3 = \alpha^2+1$, we get the desired inequality for all integers $n\geq 1$.
\end{proof}

We finally have the following identity involving $a_n$, which plays a key role in the proofs of \Cref{th:thm1} and \Cref{th:thm2}. We apply the method of constructing identities of linear recurrence sequences introduced by Zhou \cite{zhou} to prove this identity.
\begin{lemma}
\label{le:lem3}
For all integers $m\geq 3$ and $n\geq 0$, we have
\[\begin{aligned}
a_{m+n} &= a_{m-1}a_{n+2}+a_{m-3}a_{n+1} +a_{m-2}a_n\\
&= a_{m-1}a_{n+2}+(a_m-a_{m-1})a_{n+1} +a_{m-2}a_n.
\end{aligned}\]
\end{lemma}

\begin{proof}
Observe that the identity holds for $m=3$, so suppose $m\geq 4$. Consider the polynomial $h(x) = x^{m+n}-x^{m+n-1}-x^{m+n-3}$, which is divisible by the characteristic polynomial $f(x)$ of the sequence $\{a_n\}_{n\geq 0}$. Then
\[\begin{aligned}
h(x)(a_1 &+a_2x^{-1}+\cdots+a_{m-3}x^{-m+4}+a_{m-2}x^{-m+3})\\
&= a_1x^{m+n}+a_2x^{m+n-1}+a_3x^{m+n-2}+\cdots+a_{m-3}x^{n+4}+a_{m-2}x^{n+3}\\
&- a_1x^{m+n-1}-a_2x^{m+n-2}-a_3x^{m+n-3}-\cdots-a_{m-3}x^{n+3}-a_{m-2}x^{n+2}\\
&- a_1x^{m+n-3}-a_2x^{m+n-4}-a_3x^{m+n-5}-\cdots-a_{m-3}x^{n+1}-a_{m-2}x^n\\
&= a_1x^{m+n}-(a_{m-2}+a_{m-4})x^{n+2}-a_{m-3}x^{n+1}-a_{m-2}x^n\\
&\equiv 0\pmod {f(x)}.
\end{aligned}\]
Thus, in view of \cite[Thm. 2.3]{zhou}, we obtain $a_{m+n} = a_{m-1}a_{n+2}+a_{m-3}a_{n+1}+a_{m-2}a_n$.
\end{proof}

\section{Proof of \Cref{th:thm1}}
We first present the following congruence property of the Narayana numbers. 
\begin{proposition}
\label{pro:prop1}
For all integers $s\geq 1$ and $n\geq 1$, we have
\begin{align}
\label{eq:eq1}
a_{8s\cdot 3^n}  &\equiv 2s\cdot 3^{n+2} & \pmod {3^{n+3}},\nonumber\\
a_{8s\cdot 3^n+1} &\equiv 1+s\cdot 3^{n+1}+2s\cdot 3^{n+2}&\pmod {3^{n+3}},\\
a_{8s\cdot 3^n+2} &\equiv 1+ 2s\cdot 3^{n+2} &\pmod {3^{n+3}}.\nonumber
\end{align}
\end{proposition}

\begin{proof}
We first prove this for $n=1$ using induction on $s$. Note that the statement holds for the base case $s=1$, as $a_{24}\equiv 54\pmod{81}, a_{25}\equiv 64\pmod{81}$ and $a_{26}\equiv 55\pmod{81}$ by routine calculation. Suppose that the congruences \eqref{eq:eq1} hold for $s-1$. Using the recurrence formula for $a_n$, we compute  $a_{21}\equiv 63\pmod{81}, a_{22}\equiv 10\pmod{81}$ and $a_{23}\equiv-9\pmod{81}$. Applying \Cref{le:lem3}, we get
\[\begin{aligned}
a_{24s} &= a_{24+24(s-1)} = a_{23}a_{24(s-1)+2}+a_{21}a_{24(s-1)+1}+a_{22}a_{24(s-1)}\\
&\equiv -9(1+54(s-1))+63(1+63(s-1))+10(54(s-1))\pmod{81}\\
&\equiv -3969+4023s \equiv 54s\pmod{81}.
\end{aligned}\] 
Similarly, we obtain
\[\begin{aligned}
a_{24s+1} &= a_{25+24(s-1)} = a_{24}a_{24(s-1)+2}+a_{22}a_{24(s-1)+1}+a_{23}a_{24(s-1)}\\
&\equiv 54(1+54(s-1))+10(1+63(s-1))-9(54(s-1))\pmod{81}\\
&\equiv -2996+3060s \equiv 1+63s\pmod{81},\\
a_{24s+2} &= a_{26+24(s-1)} = a_{25}a_{24(s-1)+2}+a_{23}a_{24(s-1)+1}+a_{24}a_{24(s-1)}\\
&\equiv 64(1+54(s-1))-9(1+63(s-1))+54(54(s-1))\pmod{81}\\
&\equiv -5750+5805s \equiv 1+54s\pmod{81}.
\end{aligned}\] 
Thus, the congruences \eqref{eq:eq1} are true for all $s\geq 1$ and $n=1$. We now fix $s$ and show that \eqref{eq:eq1} hold using induction on $n$.  Suppose that $n\geq 2$ and the congruences \eqref{eq:eq1} are true for $n-1$. Then $a_{8s\cdot 3^{n-1}-1}\equiv -s\cdot 3^n\pmod{3^{n+2}}$ and $a_{8s\cdot 3^{n-1}-2}\equiv 1+s\cdot 3^n\pmod{3^{n+2}}$, so for some integers $c_0,c_1$ and $c_2$, we have
\[\begin{aligned}
a_{8s\cdot 3^{n-1}-2} &= 1+s\cdot 3^n+(c_1-c_0)\cdot 3^{n+2},\\
a_{8s\cdot 3^{n-1}-1} &= -s\cdot 3^n+(c_2-c_1)\cdot 3^{n+2},\\
a_{8s\cdot 3^{n-1}} &= 2s\cdot 3^{n+1}+c_0\cdot 3^{n+2},\\
a_{8s\cdot 3^{n-1}+1} &= 1+s\cdot 3^n+2s\cdot 3^{n+1}+c_1\cdot 3^{n+2},\\
a_{8s\cdot 3^{n-1}+2} &= 1+2s\cdot 3^{n+1}+c_2\cdot 3^{n+2}.
\end{aligned}\] 
Using \Cref{le:lem3}, we get
\[\begin{aligned}
a_{2(8s\cdot 3^{n-1})} &= a_{(8s\cdot 3^{n-1}+1)+(8s\cdot 3^{n-1}-1)}\\
&= a_{8s\cdot 3^{n-1}}a_{8s\cdot 3^{n-1}+1}+a_{8s\cdot 3^{n-1}-2}a_{8s\cdot 3^{n-1}}+a_{8s\cdot 3^{n-1}-1}a_{8s\cdot 3^{n-1}-1}\\
&\equiv s\cdot 3^{n+1}+(2c_0+s)\cdot 3^{n+2}\pmod {3^{n+3}},\\
a_{2(8s\cdot 3^{n-1})+1} &= a_{(8s\cdot 3^{n-1}+2)+(8s\cdot 3^{n-1}-1)}\\
&= a_{8s\cdot 3^{n-1}+1}a_{8s\cdot 3^{n-1}+1}+a_{8s\cdot 3^{n-1}-1}a_{8s\cdot 3^{n-1}}+a_{8s\cdot 3^{n-1}}a_{8s\cdot 3^{n-1}-1}\\
&\equiv 1+2s\cdot 3^n + s\cdot 3^{n+1}+(2c_1+s)\cdot 3^{n+2}\pmod {3^{n+3}},\\
a_{2(8s\cdot 3^{n-1})+2} &= a_{(8s\cdot 3^{n-1}+2)+(8s\cdot 3^{n-1})}\\
&= a_{8s\cdot 3^{n-1}+1}a_{8s\cdot 3^{n-1}+2}+a_{8s\cdot 3^{n-1}-1}a_{8s\cdot 3^{n-1}+1}+a_{8s\cdot 3^{n-1}}a_{8s\cdot 3^{n-1}}\\
&\equiv 1+s\cdot 3^{n+1}+(2c_2+s)\cdot 3^{n+2}\pmod {3^{n+3}}.
\end{aligned}\] 
Thus, applying \Cref{le:lem3} again, we arrive at
\[\begin{aligned}
a_{8s\cdot 3^n} &= a_{8s\cdot 3^{n-1}+2(8s\cdot 3^{n-1})}\\
&= a_{8s\cdot 3^{n-1}-1}a_{2(8s\cdot 3^{n-1})+2}+(a_{8s\cdot 3^{n-1}}-a_{8s\cdot 3^{n-1}-1})a_{2(8s\cdot 3^{n-1})+1}+a_{8s\cdot 3^{n-1}-2}a_{2(8s\cdot 3^{n-1})}\\
&\equiv -s\cdot 3^n+(c_2-c_1)\cdot 3^{n+2}+2s\cdot 3^{n+1}+s\cdot 3^n+(c_0+c_1-c_2)\cdot 3^{n+2}\\
&\phantom{hellohellohellohellohellohellohellohell}+ s\cdot 3^{n+1}+(2c_0+s)\cdot 3^{n+2}\pmod {3^{n+3}},\\
&\equiv 2s\cdot 3^{n+2}\pmod{3^{n+3}},\\
a_{8s\cdot 3^n+1} &= a_{2(8s\cdot 3^{n-1})+3+(8s\cdot 3^{n-1}-2)}\\
&= a_{2(8s\cdot 3^{n-1})+2}a_{8s\cdot 3^{n-1}}+a_{2(8s\cdot 3^{n-1})}a_{8s\cdot 3^{n-1}-1}+a_{2(8s\cdot 3^{n-1})+1}a_{8s\cdot 3^{n-1}-2}\\
&\equiv (2s\cdot 3^{n+1}+c_0\cdot 3^{n+2})+1+2s\cdot 3^{n+1}+(3c_1-c_0+s)\cdot 3^{n+2}\pmod {3^{n+3}},\\
&\equiv 1 + s\cdot 3^{n+1}+2s\cdot 3^{n+2}\pmod {3^{n+3}},\\
a_{8s\cdot 3^n+2} &= a_{2(8s\cdot 3^{n-1})+3+(8s\cdot 3^{n-1}-1)}\\
&= a_{2(8s\cdot 3^{n-1})+2}a_{8s\cdot 3^{n-1}+1}+a_{2(8s\cdot 3^{n-1})}a_{8s\cdot 3^{n-1}}+a_{2(8s\cdot 3^{n-1})+1}a_{8s\cdot 3^{n-1}-1}\\
&\equiv 1+s\cdot 3^n+(c_1+2c_2+2s)\cdot 3^{n+2}-s\cdot 3^n+(c_2-c_1)\cdot 3^{n+2}\pmod {3^{n+3}},\\
&\equiv 1 +2s\cdot 3^{n+2}\pmod {3^{n+3}}.
\end{aligned}\]
Hence, by induction, the congruences \eqref{eq:eq1} hold for all integers $s\geq 1$ and $n\geq 1$.
\end{proof}

We are now in a position to prove \Cref{th:thm1} by working on each case separately. 

\begin{enumerate}[label=(\alph*)]
	\item Suppose that $n\equiv k\pmod 8$ with $k\in \{1,2,3,4,6\}$. We note that the sequence $\{a_n\pmod 3\}_{n\geq 0}$ is periodic with period $8$, so $a_n\equiv a_k\pmod 3$. By routine calculation, we have $a_k\not\equiv 0\pmod 3$ for all $k\in \{1,2,3,4,6\}$, so $\nu_3(a_n) = 0$.
	\item Suppose that $n\equiv k\pmod {24}$ with $k\in \{5,7,13,15\}$. We note that the sequence $\{a_n\pmod 9\}_{n\geq 0}$ is periodic with period $24$, so $a_n\equiv a_k\pmod 9$. By routine calculation, we have $a_5=3, a_7 = 6, a_{13}=60$ and $a_{15}=129$, all of which are divisible by $3$ but not by $9$. Thus, we have $\nu_3(a_n) = 1$.
	\item Suppose $n\equiv 8\pmod{24}$. Then $n = 8s\cdot 3^m+8$ for some integers $m, s\geq 1$ with $3\nmid s$. Using the recurrence formula for $a_n$ and \Cref{pro:prop1}, we deduce that $a_n\equiv 9\pmod {3^{m+3}}$. Thus, we get $\nu_3(a_n) = 2$. 
	\item Suppose $n\equiv 23\pmod{24}$. Then $n = 8s\cdot 3^m-1$ for some integers $m, s\geq 1$ with $3\nmid s$, so that $\nu_3(n+1) = m$. Using the recurrence formula for $a_n$ and \Cref{pro:prop1}, we deduce that $a_n\equiv -s\cdot 3^{m+1}\pmod {3^{m+3}}$. Thus, we get $\nu_3(a_n) = m+1 = \nu_3(n+1)+1$.
	\item Suppose $n\equiv 21\pmod{24}$. Then $n = 8s\cdot 3^m-3$ for some integers $m, s\geq 1$ with $3\nmid s$, so that $\nu_3(n+3) = m$. Using the recurrence formula for $a_n$ and \Cref{pro:prop1}, we deduce that $a_n\equiv (2s\cdot 3^{m+2}+s\cdot 3^{m+1})\pmod {3^{m+3}}$. Thus, we get $\nu_3(a_n) = m+1 = \nu_3(n+3)+1$.
	\item Suppose $n\equiv 0\pmod{24}$. Then $n = 8s\cdot 3^m$ for some integers $m, s\geq 1$ with $3\nmid s$, so that $\nu_3(n) = m$. By \Cref{pro:prop1}, we have $a_n\equiv 2s\cdot 3^{m+2}\pmod {3^{m+3}}$. Thus, we get $\nu_3(a_n) = m+2 = \nu_3(n)+2$.
	\item Suppose $n\equiv 16\pmod{24}$. Then $n = 8s\cdot 3^m-8$ for some integers $m, s\geq 1$ with $3\nmid s$, so that $\nu_3(n+8) = m$. Using the recurrence formula for $a_n$ and \Cref{pro:prop1}, we deduce that $a_n\equiv -2s\cdot 3^{m+2}\pmod {3^{m+3}}$. Thus, we get $\nu_3(a_n) = m+2 = \nu_3(n+8)+2$.
\end{enumerate}
This completes the proof of \Cref{th:thm1}.

\section{Proof of \Cref{th:thm2}}

We note that if $m\leq 5$, then the only solutions are the ones listed in \Cref{th:thm2}, so we now suppose that $m\geq 6$. Applying \Cref{th:thm1} and \Cref{le:lem1} with $p=3$, we obtain
\[\begin{aligned}
\dfrac{m}{2}-\left\lfloor\dfrac{\log m}{\log 3}\right\rfloor-1 &\leq \nu_3(a_n)\leq\nu_3(n(n+1)(n+3)(n+8))+6 \leq 4\nu_3(n+\delta)+6,
\end{aligned}\] 
for some $\delta\in \{0,1,3,8\}$. This implies that
\[3^{\lfloor m/8 - \lfloor (\log m)/(\log 3)\rfloor/4 - 7/4\rfloor}\leq 3^{\nu_3(n+\delta)}\leq n+\delta\leq n+8\]
and taking logarithms leads to 
\begin{align}
\label{eq:eq2}
\left\lfloor \dfrac{m}{8} - \dfrac{1}{4}\left\lfloor\dfrac{\log m}{\log 3}\right\rfloor -\dfrac{7}{4}\right\rfloor\leq \dfrac{\log(n+8)}{\log 3}.
\end{align}
From \Cref{le:lem2}, we have $\alpha^{n-3} \leq a_n = m! < (m/2)^m$, so that $n < 3+m\log(m/2)/\log \alpha$. Plugging this in \cref{eq:eq2} yields
\[\left\lfloor \dfrac{m}{8} - \dfrac{1}{4}\left\lfloor\dfrac{\log m}{\log 3}\right\rfloor -\dfrac{7}{4}\right\rfloor\leq \dfrac{\log(11+m\log(m/2)/\log \alpha)}{\log 3}.\]
Thus, we deduce that $m\leq 68$ and $n \leq 630$. A simple computational search using \textit{Mathematica} shows that there are no solutions in the range $m\in [6,68]$ and $n\in [1,630]$. This completes the proof of \Cref{th:thm2}.

\bibliography{narayanabib}
\bibliographystyle{amsplain}

%\begin{theorem}
%	For integers $n\geq 1$, we have
%	\[\nu_2(a_{2,n}) = \begin{cases}
%	0, &\text{ if }n\equiv 1\pmod 3;\\
%	1, &\text{ if }n\equiv 2\pmod 6;\\
%	\nu_2(n)+1, &\text{ if }n\equiv 0\pmod 6;\\
%	\nu_2(n+1)+1, &\text{ if }n\equiv 5\pmod 6;\\
%	\nu_2(n+3)+1, &\text{ if }n\equiv 3\pmod 6.\\
%	\end{cases}\] 
%\end{theorem}
%
%\begin{theorem}
%	The only solutions to the Diophantine equation $a_{k,n} = m!$ in positive integers $k,n$ and $m$ with $k\leq 3$ are
%	\[(k,n,m)\in \{(1,1,1),(1,2,1),(1,3,1),(1,4,2),(1,7,3),(2,1,1),(2,2,2),(3,1,1)\}.\]
%\end{theorem}

\end{document}